\documentclass[a4paper,12pt]{article}

\usepackage[left=2cm,right=2cm, top=2cm,bottom=2cm,bindingoffset=0cm]{geometry}

\usepackage{verbatim}
\usepackage{amsmath}
\usepackage{amsthm}
\usepackage{amssymb}
\usepackage{delarray}
\usepackage{cite}
\usepackage{hyperref}

\newcommand{\de}{\delta}
\newcommand{\la}{\lambda}

\newcommand{\vv}{\varphi}
\newcommand{\iy}{\infty}

\newcommand{\bu}{\bullet}

\newtheorem{thm}{Theorem}
\newtheorem{lem}{Lemma}
\newtheorem*{lem3}{Lemma 3$^*$}

 \allowdisplaybreaks

\begin{document}

\begin{center}
{\large\bf
An inverse problem for the matrix quadratic pencil

on a finite interval}
\\[0.2cm]
{\bf Natalia Bondarenko} \\[0.2cm]
\end{center}

\vspace{0.5cm}

{\bf Abstract.} We consider a quadratic matrix boundary value problem with
equations and boundary conditions dependent on a spectral parameter. 
We study an inverse problem that consists in recovering the differential pencil 
by the so--called Weyl matrix. We obtain asymptotic formulas for the solutions of the considered
matrix equation. Using the ideas of the method of spectral mappings, we prove
the uniqueness theorem for this inverse problem.

\medskip


\vspace{1cm}

{\bf 1. Introduction and main results} \\

In this paper, we consider the boundary value problem $L = L(\ell, U, V)$ for the equation
\begin{equation} \label{eqv}
    \ell Y := Y'' + (\rho^2 \cdot I + 2 i \rho Q_1(x) + Q_0(x)) Y = 0, \quad x \in (0, \pi), 
\end{equation}
with the boundary conditions
\begin{equation} \label{BC}
    \begin{array}{l}
        U(Y) := Y'(0) + (i \rho h_1 + h_0) Y(0) = 0, \\
        V(Y) := Y'(\pi) + (i \rho H_1 + H_0) Y(\pi) = 0.
    \end{array}
\end{equation}
Here $Y(x) = [y_k(x)]_{k = \overline{1, m}}$ is a column vector, $\rho$ is the spectral parameter,
$I$ is the $m \times m$ unit matrix,
$Q_s(x) = [Q_{s, jk}(x)]_{j,k = \overline{1, m}}$ are $m \times m$ matrices with entries
$Q_{s, jk}(x) \in W_1^s[0, \pi]$, $s = 0, 1$
$h_s = [h_{s, j k}]_{j, k = \overline{1, m}}$, 
$H_s = [H_{s, j k}]_{j, k = \overline{1, m}}$,
where $h_{s, j k}$, $H_{s, j k}$ are complex numbers.

We assume that $\det (I \pm h_1) \ne 0$
and $\det (I \pm H_1) \ne 0$. This condition excludes problems of Regge type (see \cite{Yur84}) from consideration,
as they require a separate investigation.

Differential equations with nonlinear dependence on the spectral parameter,
or with so-called ``energy--dependent'' coefficients, frequently appear in
mathematics and applications (see \cite{Kel71, KS83, Shkal83, Markus86, Yur97} and references therein).
In particular, inverse problems for such equations arise in investigation of mathematical models
describing collisions of relativistic spinless particles \cite{JJ72} or
proper vibrations of a string with viscous drag \cite{Yam90}.

In this paper, we investigate the inverse problem for the pencil $L$,
which consists in recovering coefficients of the boundary value problem 
\eqref{eqv}, \eqref{BC} by its spectral characteristics.
In the scalar case ($m = 1$) inverse problems for quadratic pencils were studied
in works \cite{GG81, Yur00, BU06, BU12, HP12, Pron12}.

In the particular case when $Q_1(x) \equiv 0$, equation \eqref{eqv}
becomes the matrix Sturm--Liouville equation. 
In recent years, significant progress has been made in the inverse problems theory for this equation.
Constructive algorithms for solution of inverse problems were suggested, and characterization of spectral data
was given (see \cite{Yur06, CK09, MT10, Bond11, Bond12}). 

Now we plan to use ideas, developed for
the matrix Sturm--Liouville equation, for problems with nonlinear dependence on the spectral parameter.
In the present paper, we consider a general situation, without any conditions of selfadjointness on the coefficients
and with the spectral parameter in the boundary conditions. We study the inverse problem 
for the pencil $L$ by so--called Weyl matrix and prove the uniqueness theorem for the solution of this problem.
For our investigation, we develop ideas of the method of spectral mappings \cite{FY01, Yur02}.
This method also can be used to obtain a constructive procedure for the solution
of this inverse problem, but this question requires separate investigation. 

One of the main difficulties in problems for differential pencils is related to asymptotic behavior of solutions.
The main terms of asymptotic representations depend on the coefficients of the pencil $Q_1$, $h_1$, $H_1$.
Derivation of asymptotic formulas is nontrivial even in the scalar case, and in the matrix case it is more complicated,
because of additional difficulties connected with noncommutativity of matrix multiplication.
We obtain the asymptotics for the fundamental system of solutions for equation \eqref{eqv}. They are required to
prove uniqueness, but they also can be considered as a separate result.
 
Now proceed to the formulation of the main result.
Let $\Phi(x, \rho) = [\Phi_{jk}(x, \rho)]_{j, k = \overline{1, m}}$ 
be the matrix solution of equation \eqref{eqv} satisfying the conditions
$U(\Phi) = I$, $V(\Phi) = 0$. We call $\Phi(x, \rho)$ the
\textit{Weyl solution} for $L$. Put $M(\rho) := \Phi(0, \rho)$. 
The matrix $M(\rho) = [M_{jk}(\rho)]_{j,k = \overline{1, m}}$ is called
the \textit{Weyl matrix} for $L$. The notion of the Weyl matrix
is a generalization of the notion of the Weyl function ($m$-function)
for the scalar case (see \cite{Mar77, FY01}) and the notion of the 
Weyl matrix for the matrix Sturm--Liouville operator (see \cite{Bond11, Bond12}).

The inverse problem is stated as follows.

\medskip

{\bf Inverse Problem 1.} {\it Given a Weyl matrix $M(\rho)$, construct the coefficients of 
the pencil $L$.}

\medskip

{\it Remark.} One can prove in a standard way that
the boundary problem $L$ has a countable set of eigenvalues $\{ \rho_n \}$.   
The entries of $M(\rho)$ are meromorphic in $\rho$
and their poles coincide with $\{\rho_n\}$. As in the scalar case (see \cite{BU06}),
the following representation is valid
\begin{equation} \label{Mseries}
 	M(\rho) = \sum_n \sum_{\nu = 1}^{m_n} \frac{M_{n \nu}}{(\rho - \rho_n)^{\nu}},
\end{equation}
where $m_n$ are multiplicities of the corresponding eigenvalues $\rho_n$, and
$M_{n \nu}$ are some matrix coefficients. 

Following \cite{BU06, BU12}, we call the collection $\{ \rho_n, M_{n \nu}\}$ the {\it spectral data}
of the pencil $L$. 
By virtue of \eqref{Mseries}, the spectral 
data determine the Weyl matrix uniquely.
Therefore, Inverse Problem~1 is equivalent to

\medskip

{\bf Inverse Problem 2.} {Given spectral data $\{ \rho_n, M_{n \nu} \}$, 
construct the coefficients of the pencil $L$.}

\medskip

In this paper, we restrict ourselves to Inverse Problem~1.

Along with $L$ we consider a pencil $\tilde L$ of the same form but with 
other coefficients $\tilde Q_s(x)$, $\tilde h_s$, $\tilde H_s$. We agree that if a symbol 
$\gamma$ denotes an object related to $L$ then $\tilde \gamma$
denotes the corresponding object related to $\tilde L$.

We now state the uniqueness theorem for Inverse Problem 1.

\begin{thm}               	
If $M(\rho) = \tilde M(\rho)$, then $L = \tilde L$.
Hence the Weyl matrix determines the coefficients of the pencil \eqref{eqv}, \eqref{BC} uniquely.
\end{thm}

\medskip

In order to prove Theorem~1, we need asymptotics for the solutions of equation \eqref{eqv}.
They are obtained in Section 2. In Section 3, we provide the proof of the uniqueness theorem.

\bigskip

{\bf 2. Asympotic behavior of the solutions}\\

Let $C(x, \rho) = [C_{jk}(x, \rho)]_{j, k = \overline{1, m}}$ and
$S(x, \rho) = [S_{jk}(x, \rho)]_{j, k = \overline{1, m}}$ 
be the matrix solutions of equation \eqref{eqv}
under the initial conditions
$$
  	C(0, \rho) = S'(0, \rho) = I, \quad C'(0, \rho) = S(0, \rho) = 0.
$$
The main goal of this section is to obtain the asymptotics of $S(x, \rho)$ and $C(x, \rho)$
as $|\rho| \to \infty$.

Let the matrix functions $P_{+}(x)$, $P_{-}(x)$, $P_{+}^*(x)$ and $P_{-}^*(x)$ be the solutions of
the Cauchy problems
\begin{equation} \label{cauchyP}
\begin{array}{l}
	P_{\pm}'(x) = \pm Q_1(x) P_{\pm}(x), \quad P_{\pm}(0) = I, \\
	P_{\pm}^{*'}(x) = \pm P_{\pm}^*(x) Q_1(x), \quad P_{\pm}^*(0) = I.	
\end{array}
\end{equation}

\medskip

{\it Remark.} In the scalar case $(m = 1)$, one has 
$$
 	P_{\pm}(x) = P^*_{\pm}(x) = \exp \left\{ \pm \int_0^x Q_1(t) \, dt \right\}.
$$

\medskip

\begin{lem} \label{lem:P}
The following relations hold
$$
 	P_{+}(x) P_{-}^*(x) = P_{-}^*(x) P_{+}(x) = I, \quad P_{-}(x) P_{+}^*(x) = P_{+}^*(x) P_{-}(x) = I, 
$$
for all $x \in [0, \pi]$.
\end{lem}

\begin{proof}
Using \eqref{cauchyP}, we get
$$
 	(P_{-}^*(x) P_{+}(x))' = (P_{-}^*(x))' P_{+}(x) + P_{-}^{*}(x) P_{+}'(x) =
 	- P_{-}^*(x) Q_1(x)  P_{+}(x) + P_{-}^*(x) Q_1(x)  P_{+}(x) = 0.
$$ 
Hence $P_{-}^*(x) P_{+}(x)$ does not depend on $x$, so $P_{-}^*(x) P_{+}(x) = P_{-}^*(0) P_{+}(0) = I$.
The other relations can be proved similarly.
\end{proof}

\medskip

\begin{thm} \label{thm:asympt}
For $x \in [0, \pi]$, $|\rho| \to \infty$, $\nu = 0, 1$, the following relations hold
\begin{gather} \label{asymptC}
C^{(\nu)}(x, \rho) = \frac{(i \rho)^{\nu}}{2} \exp(i \rho x) P_{-}(x) + \frac{(- i \rho)^{\nu}}{2} \exp(-i\rho x) P_{+}(x)
 + O(\rho^{\nu-1} \exp(|\tau|x)), \\ \nonumber
S^{(\nu)}(x, \rho)  = \frac{(i \rho)^{\nu - 1}}{2} \exp(i \rho x) P_{-}(x)  + \frac{(-i \rho)^{\nu - 1}}{2}
\exp(- i \rho x) P_{+}(x) + O(\rho^{\nu-2} \exp(|\tau|x)),
\end{gather}
where $\tau := \mbox{Im} \, \rho$.
\end{thm}

\medskip

In order to prove the theorem, we develop the ideas of \cite{But11} with necessary modifications caused by the matrix case. 

\begin{proof}
{\bf 1.} First we derive Volterra integral equations for $S(x, \rho)$ and $C(x, \rho)$.

One can easily check that the matrix functions
\begin{gather*}
 	C_0(x, \rho) = \frac{1}{2} \exp(i \rho x) P_{-}(x) + \frac{1}{2} \exp(- i \rho x) P_{+}(x), \\ 
 	S_0(x, \rho) = \frac{1}{2 i \rho} \exp(i \rho x) P_{-}(x) - \frac{1}{2 i \rho} \exp(-i \rho x) P_{+}(x)
\end{gather*}
form a fundamental system of solutions for the differential equation
$$
Y'' + Q'_1(x) (i \rho \cdot I - Q_1(x))^{-1} Y' + (\rho^2 \cdot I + 2 i \rho Q_1(x) - Q_1^2(x)) Y = 0.
$$
Rewrite equation \eqref{eqv} in the form
\begin{gather} \label{eqvrewr}
Y'' + Q'_1(x) (i \rho \cdot I - Q_1(x))^{-1} Y' + (\rho^2 \cdot I + 2 i \rho Q_1(x) - Q_1^2(x)) Y = F(x, \rho, Y), \\
\label{defF}
F(x, \rho, Y) := Q'_1(x) (i \rho \cdot I - Q_1(x))^{-1} Y' - (Q_1^2(x) + Q_0(x))Y.
\end{gather}
Apply the method of variation of parameters to this equation. Every solution of \eqref{eqvrewr} can be represented in 
the form 
$$
	Y(x, \rho) = C_0(x, \rho) A(x, \rho) + S_0(x, \rho) B(x, \rho),
$$
where coefficient matrices $A(x, \rho)$ and $B(x, \rho)$ can be found from the system
\begin{equation} \label{systemAB}
 	\left[\begin{array}{cc} C_0(x, \rho) & S_0(x, \rho) \\ C'_0(x, \rho) & S'_0(x, \rho) \end{array}\right] \cdot
 	\left[\begin{array}{c} A'(x, \rho) \\ B'(x, \rho) \end{array}\right] = 
 	\left[\begin{array}{c} 0 \\ F(x, \rho, Y) \end{array}\right]
\end{equation}
and the initial conditions 
\begin{equation} \label{initAB} 
 	A(0, \rho) = Y(0, \rho), \quad B(0, \rho) = (I - Q_1(0) / (i \rho))^{-1} Y'(0, \rho).
\end{equation}

In order to solve system \eqref{systemAB}, we introduce the matrix functions
\begin{gather*}
 	C_0^*(x, \rho) = \frac{1}{2} \exp(i \rho x) P_{-}^*(x) + \frac{1}{2} \exp(- i \rho x) P_{+}^*(x), \\ 
 	S_0^*(x, \rho) = \frac{1}{2 i \rho} \exp(i \rho x) P_{-}^*(x) - \frac{1}{2 i \rho} \exp(-i \rho x) P_{+}^*(x).
\end{gather*}
Using Lemma~\ref{lem:P}, one can easily show that
$$
 	\left[\begin{array}{cc} C_0 & S_0 \\ C'_0 & S'_0 \end{array}\right]^{-1} =
 	\left[\begin{array}{cc} S_0^{*'} & -S_0^* \\ -C_0^{*'} & C_0^* \end{array}\right] \cdot
 	\left[\begin{array}{cc} i \rho (i \rho \cdot I - Q_1)^{-1} & 0 \\ 0 & i \rho (i \rho \cdot I - Q_1)^{-1} \end{array}\right]. 
$$
Therefore
\begin{gather*}
 	A'(x, \rho) = - i \rho S_0^*(x, \rho) (i \rho \cdot I - Q_1(x))^{-1} F(x, \rho, Y),\\
 	B'(x, \rho) = i \rho C_0^*(x, \rho) (i \rho \cdot I - Q_1(x))^{-1} F(x, \rho, Y).
\end{gather*}
and
\begin{multline} \label{inteqY}
 	Y(x, \rho) = C_0(x, \rho) A(0, \rho) + S_0(x, \rho) B(0, \rho) + \\ +
 	i \rho \int_0^x (S_0(x, \rho) C_0^*(t, \rho) - C_0(x, \rho) S_0^*(t, \rho))
 	(i \rho \cdot I - Q_1(t))^{-1} F(t, \rho, Y) \, dt.
\end{multline}
Simple calculations show that
$$
 	S_0(x, \rho) C_0^*(t, \rho) - C_0(x, \rho) S_0^*(t, \rho) = \frac{1}{2 i \rho} \bigl\{ \exp(i \rho (x - t)) P_{-}(x) P_{+}^*(t)
 	- \exp(- i \rho(x - t)) P_{+}(x) P_{-}^*(t) \bigr\}.
$$
Substituting $Y = C(x, \rho)$ and $Y = S(x, \rho)$ into \eqref{initAB} and \eqref{inteqY}, we arrive at
the following integral equations
\begin{multline} \label{inteqC}
 	C(x, \rho) = C_0(x, \rho) + \frac{1}{2} \int_0^x
 	\bigl\{ \exp(i \rho (x - t)) P_{-}(x) P_{+}^*(t) - \\
 	- \exp(- i \rho(x - t)) P_{+}(x) P_{-}^*(t) \bigr\} (i \rho \cdot I - Q_1(t))^{-1} F(t, \rho, C) \, dt,
\end{multline} 
\begin{multline*}
 	S(x, \rho) = S_0(x, \rho) \left(I - \frac{Q_1(0)}{i \rho}\right)^{-1} +
 	\frac{1}{2} \int_0^x
 	\bigl\{ \exp(i \rho (x - t)) P_{-}(x) P_{+}^*(t)
 	- \\ - \exp(- i \rho(x - t)) P_{+}(x) P_{-}^*(t)\bigr\} (i \rho \cdot I - Q_1(t))^{-1} F(t, \rho, S) \, dt,
\end{multline*}
where $F(x, \rho, Y)$ is defined in \eqref{defF}.

\smallskip

{\bf 2.} Then we continue to work with $C(x, \rho)$. The function $S(x, \rho)$ can be treated similarly.
Differentiating \eqref{inteqC} with respect to $x$ and using \eqref{cauchyP}, we get
\begin{multline} \label{inteqCp}
 	C'(x, \rho) = C'_0(x, \rho) + \frac{1}{2} (i \rho \cdot I - Q_1(x)) \int_0^x
 	\bigl\{ \exp(i \rho (x - t)) P_{-}(x) P_{+}^*(t) + \\
 	+ \exp(- i \rho(x - t)) P_{+}(x) P_{-}^*(t) \bigr\} (i \rho \cdot I - Q_1(t))^{-1} F(t, \rho, C) \, dt, 		
\end{multline}
Denote
$$
 	\mu_{\nu}(\rho) := \max_{x \in [0, \pi]} \| C^{(\nu)}(x, \rho) \exp(-|\tau| x) \|, \quad \nu = 0, 1.
$$

We agree to denote by the same symbol $K$ different positive constants not depending on $x$ and $\rho$, and
we mean
by $\| . \|$ the following matrix norm: $\| A \| = \max_{j = \overline{1, m}} \sum_{k = 1}^m |a_{ij}|$,
$A = [a_{jk}]_{j, k = \overline{1, m}}$.

Since 
$$
 	\| F(t, \rho, C) \| \le K \| Q'_1(t) \| \frac{\mu_1(\rho)}{|\rho|} + \| Q_1^2(t) + Q_0(t) \| \mu_0(\rho),
$$
we get from \eqref{inteqC} and \eqref{inteqCp}
$$
 	\mu_0(\rho) \le K \left( 1 + \frac{\mu_0(\rho)}{|\rho|} + \frac{\mu_1(\rho)}{|\rho|^2}\right), \quad
 	\mu_1(\rho) \le K \left( |\rho| + \mu_0(\rho) + \frac{\mu_1(\rho)}{|\rho|}\right), 
$$
whence we obtain $\mu_{\nu}(\rho) \le K |\rho|^{\nu}$ or
$$
 	C^{(\nu)}(x, \rho) = O(|\rho|^{\nu}\exp(|\tau|x)). 
$$
Substituting this into \eqref{inteqC} and \eqref{inteqCp}, we arrive at \eqref{asymptC}.
\end{proof}

\bigskip

{\bf 3. Proof of the uniqueness theorem}

\medskip

Let $\vv(x, \rho) = [\vv_{jk}(x, \rho)]_{j, k = \overline{1, m}}$ and
$\psi(x, \rho) = [\psi_{jk}(x, \rho)]_{j, k = \overline{1, m}}$
be the matrix solutions of equation \eqref{eqv}
under the initial conditions
$$
    \vv(0, \rho) = \psi(\pi, \rho) = I, \quad U(\vv) = V(\psi) = 0.
$$
It is easy to check that
\begin{equation} \label{Phiexp}
    \Phi(x, \rho) = S(x, \rho) + \vv(x, \rho) M(\rho),
\end{equation}
\begin{equation} \label{Upsi}
    \Phi(x, \rho) = \psi(x, \rho) (U(\psi))^{-1}.
\end{equation}
 
We expand $\vv(x, \rho)$ by the fundamental system of solutions $C(x, \rho)$ and $S(x, \rho)$
$$
 	\vv(x, \rho) = C(x, \rho) - S(x, \rho) (i \rho h_1 + h_0) 
$$
and, using Theorem~\ref{thm:asympt}, we get 
\begin{multline*}
 	\vv^{(\nu)}(x, \rho) = \frac{(i \rho)^{\nu}}{2} \exp(i \rho x) P_{-}(x) (I - h_1) + 
 	\frac{(- i \rho)^{\nu}}{2} \exp(- i \rho x) P_{+}(x) (I + h_1) + \\ + O(\rho^{\nu -1} \exp(|\tau|x)), \quad
 	x \in [0, \pi], \, \nu = 0, 1,  \, |\rho| \to \infty.
\end{multline*}

In order to obtain the asymptotics for $\psi(x, \rho)$, we can apply the substitution $x \to \pi - x$,
$h_1 \to -H_1$, $h_0 \to -H_0$. Then we need the analogs of $P_{\pm}(x)$ and $P_{\pm}^*(x)$, which we denote by
$P_{\pm}^{\bullet}(x)$ and $P_{\pm}^{\bullet*}(x)$ and define as the solutions of the Cauchy problems
\begin{gather*}
 	P_{\pm}^{\bullet '}(x) = \mp Q_1(x) P_{\pm}^{\bullet}(x), \quad P_{\pm}^{\bullet}(\pi) = I, \\
 	P_{\pm}^{\bullet *'}(x) = \mp P_{\pm}^{\bullet*}(x) Q_1(x), \quad P_{\pm}^{\bullet*}(\pi) = I.
\end{gather*}

The following lemma establishes connections between these asymptotic coefficients.

\medskip

\begin{lem} \label{lem:P2}
For $x \in [0, \pi]$, the following relations hold
$$
    P_{+}(x) P_{-}^{\bu}(0) = P_{-}^{\bu}(x), \quad P_{-}(x) P_{+}^{\bu}(0) = P_{+}^{\bu}(x),
$$
$$
    P_{-}^{\bu *}(0) P^*_{+}(x) = P_{-}^{\bu *}(x), \quad P_{+}^{\bu *}(0) P^*_{-}(x) = P_{+}^{\bu *}(x).
$$
\end{lem}

\medskip

The proof, based on using the corresponding Cauchy problems, is trivial.

One can obtain the following asymptotic formula
\begin{multline*}
 	\psi^{(\nu)}(x, \rho) = \frac{(- i \rho)^{\nu}}{2} \exp(i \rho (\pi - x)) P^{\bu}_{-}(x)(I + H_1) + 
 	\frac{(i \rho)^{\nu}}{2} \exp(- i \rho (\pi - x)) P^{\bu}_{+}(x)(I - H_1) + \\ +
 	O(\rho^{\nu -1} \exp(|\tau| (\pi - x))),
 	\quad x \in [0, \pi], \, \nu = 0, 1, \, |\rho| \to \infty,
\end{multline*}
and then using \eqref{Upsi}, one can derive asymptotic formulas for $\Phi(x, \rho)$ and $M(\rho)$.
	
Hereafter the asymptotics in the angles
$\Theta^{\pm}_{\delta} := \{ \rho \in \mathbb{C} \colon \de \le \pm \arg \rho \le \pi - \de \}$, $0 < \de < \pi$
will be required, so we formulate the following result.

\medskip

\begin{lem} \label{lem:asympt}
Suppose $x \in (0, \pi)$, $\nu = 0, 1$, $|\rho| \to \iy$; then

(i) for $\rho \in \Theta^+_{\de}$, we have
\begin{align*}
    \vv^{(\nu)}(x, \rho) & = \frac{(- i \rho)^{\nu}}{2} \exp(- i \rho x) P_{+}(x)(I + h_1) + O(\rho^{\nu -1} \exp(|\tau| x)), \\
    \psi^{(\nu)}(x, \rho) & = \frac{(i \rho)^{\nu}}{2} \exp(- i \rho (\pi - x)) P^{\bu}_{+}(x)(I - H_1) + O(\rho^{\nu -1} \exp(|\tau| (\pi - x))), \\
    \Phi^{(\nu)}(x, \rho) & = -(- i \rho)^{\nu -1} \exp(i \rho x) P_{+}^{\bu}(x) (P_{+}^{\bu}(0))^{-1} (I + h_1)^{-1} + O(\rho^{\nu - 2} \exp(-|\tau| x)), \\
    M(\rho) & = (i \rho)^{-1} (I + h_1)^{-1} + O(\rho^{-2}).
\end{align*}

(ii) for $\rho \in \Theta^-_{\de}$, we have
\begin{align*}
    \vv^{(\nu)}(x, \rho) & = \frac{(i \rho)^{\nu}}{2} \exp(i \rho x) P_{-}(x)(I - h_1) + O(\rho^{\nu -1} \exp(|\tau| x)), \\
    \psi^{(\nu)}(x, \rho) & = \frac{(- i \rho)^{\nu}}{2} \exp(i \rho (\pi - x)) P^{\bu}_{-}(x)(I + H_1) + O(\rho^{\nu -1} \exp(|\tau| (\pi - x))), \\
    \Phi^{(\nu)}(x, \rho) & = -(i \rho)^{\nu -1} \exp(- i \rho x) P_{-}^{\bu}(x) (P_{-}^{\bu}(0))^{-1} (I - h_1)^{-1} + O(\rho^{\nu - 2} \exp(-|\tau| x)), \\
    M(\rho) & = - (i \rho)^{-1} (I - h_1)^{-1} + O(\rho^{-2}).
\end{align*}
\end{lem}
 
\medskip

By definition, put
$$
    \ell^* Z := Z'' + Z (\rho^2 \cdot I + 2 i \rho Q_1(x) + Q_0(x)),
$$
$$
    U^*(Z) := Z'(0) + Z(0) (i \rho h_1 + h_0),
$$
$$
    V^*(Z) := Z'(\pi) + Z(\pi) (i \rho H_1 + H_0),
$$
$$
    \langle Z, Y \rangle = Z'(x) Y(x) - Z(x) Y'(x),
$$
where $Z = [z_k]^t_{k = \overline{1, m}}$ is a row vector ($t$ is the sign for the transposition). Then
\begin{equation} \label{smEq2}
    \langle Z, Y \rangle_{|x = 0} = U^*(Z) Y(0) - Z(0) U(Y), \quad
    \langle Z, Y \rangle_{|x = \pi} = V^*(Z) Y(\pi) - Z(\pi) V(Y).    
\end{equation}
If $Y(x, \rho)$ and $Z(x, \rho)$ satisfy the equations $\ell Y(x, \rho) = 0$ and $\ell^* Z(x, \rho) = 0$
respectively, then
\begin{equation} \label{smEq1}
    \frac{d}{dx} \langle Z(x, \rho), Y(x, \rho) \rangle = 0.
\end{equation} 
Let $\vv^*(x, \rho)$, $S^*(x, \rho)$, $\psi^*(x, \rho)$ and $\Phi^*(x, \rho)$ be the matrices, satisfying 
the equation $\ell^* Z = 0$ and the conditions 
$\vv^*(0, \rho) = {S^*}'(0, \rho) = \psi^*(\pi, \rho) = U^*(\Phi^*) = I$,
$U^*(\vv^*) = S^*(0, \rho) = V^*(\psi^*) = V^*(\Phi^*) = 0$. Put $M^*(\rho) := \Phi^*(0, \rho)$.
Then
\begin{equation} \label{Phiexp*}
    \Phi^*(x, \rho) = S^*(x, \rho) + M^*(\rho) \vv^*(x, \rho) = (U^*(\psi^*))^{-1} \psi^*(x, \rho).
\end{equation}
According to \eqref{smEq1}, $\langle \Phi^*(x, \rho), \Phi(x, \rho) \rangle$ does not depend on $x$.
Using \eqref{smEq2}, we get
$$
    \langle \Phi^*(x, \rho), \Phi(x, \rho) \rangle_{|x = 0} = M(\rho) - M^*(\rho), \quad
    \langle \Phi^*(x, \rho), \Phi(x, \rho) \rangle_{|x = \pi} = 0.        
$$
Therefore, $M(\rho) \equiv M^*(\rho)$.

Using \eqref{smEq2}, one can easily show that
$$
\left[\begin{array}{ll} {\Phi^*}'(x, \rho) & -\Phi^*(x,\rho)\\ -{\vv^*}'(x,\rho) & \vv^*(x,\rho) \end{array}\right]    
\left[\begin{array}{ll} \vv(x,\rho) & \Phi(x,\rho)\\ \vv'(x,\rho) & \Phi'(x,\rho) \end{array}\right] =
\left[\begin{array}{ll} I & 0 \\ 0  & I \end{array}\right].
$$
Hence,
\begin{equation} \label{inverseform}
\left[\begin{array}{ll} \vv(x,\rho) & \Phi(x,\rho)\\ \vv'(x,\rho) & \Phi'(x,\rho) \end{array}\right]^{-1} =
\left[\begin{array}{ll} {\Phi^*}'(x, \rho) & -\Phi^*(x,\rho)\\ -{\vv^*}'(x,\rho) & \vv^*(x,\rho) \end{array}\right].    
\end{equation}

We also need the asymptotics for $\vv^*(x, \rho)$ and $\Phi^*(x, \rho)$.

\begin{lem3}
Suppose $x \in (0, \pi)$, $\nu = 0, 1$, $|\rho| \to \iy$; then

(i) for $\rho \in \Theta^+_{\de}$, we have
\begin{align*}
    {\vv^*}^{(\nu)}(x, \rho) & = \frac{(- i \rho)^{\nu}}{2} \exp(- i \rho x) (I + h_1) P^*_{+}(x) + O(\rho^{\nu -1} \exp(|\tau| x)), \\
    {\psi^*}^{(\nu)}(x, \rho) & = \frac{(i \rho)^{\nu}}{2} \exp(- i \rho (\pi - x)) (I - H_1) P^{\bu *}_{+}(x) + O(\rho^{\nu -1} \exp(|\tau| (\pi - x))), \\ 
    {\Phi^*}^{(\nu)}(x, \rho) & = -(- i \rho)^{\nu -1} \exp(i \rho x) (I + h_1)^{-1} (P_{+}^{\bu *}(0))^{-1} P_{+}^{\bu *}(x) + O(\rho^{\nu - 2} \exp(-|\tau| x)).
\end{align*}

(ii) for $\rho \in \Theta^-_{\de}$, we have
\begin{align*}
    {\vv^*}^{(\nu)}(x, \rho) & = \frac{(i \rho)^{\nu}}{2} \exp(i \rho x) (I - h_1) P^*_{-}(x) + O(\rho^{\nu -1} \exp(|\tau| x)), \\
    {\psi^*}^{(\nu)}(x, \rho) & = \frac{(- i \rho)^{\nu}}{2} \exp(i \rho (\pi - x)) (I + H_1) P^{\bu *}_{-}(x) + O(\rho^{\nu -1} \exp(|\tau| (\pi - x))), \\ 
    {\Phi^*}^{(\nu)}(x, \rho) & = -(i \rho)^{\nu -1} \exp(- i \rho x) (I - h_1)^{-1} (P_{-}^{\bu *}(0))^{-1} P_{-}^{\bu *}(x) + O(\rho^{\nu - 2} \exp(-|\tau| x)).
\end{align*}
\end{lem3}

\begin{proof}[Proof of Theorem 1.]
Consider the problems $L$ and $\tilde L$ with the Weyl matrices $M(\la) \equiv \tilde M(\la)$.
Note that according to asymptotics for $M(\la)$ of Lemma~\ref{lem:asympt}, we straightway get 
\begin{equation} \label{eqh1}
h_1 = \tilde h_1.
\end{equation}

Now we consider the block-matrix $\mathcal{P}(x,\rho)=[\mathcal{P}_{jk}(x,\rho)]_{j,k=1,2}$
defined by
\begin{equation} \label{defP}
\mathcal{P}(x,\rho) \left[ \begin{array}{ll} \tilde\vv(x,\rho) & \tilde\Phi(x,\rho)\\ \tilde\vv'(x,\rho) & \tilde\Phi'(x,\rho) \end{array}\right]
= \left[\begin{array}{ll} \vv(x,\rho) & \Phi(x,\rho)\\ \vv'(x,\rho) & \Phi'(x,\rho) \end{array}\right].             
\end{equation}
Taking \eqref{inverseform} into account, we calculate
\begin{equation} \label{Pj12}
\begin{array}{l}
\mathcal{P}_{j1}(x,\rho)=\vv^{(j-1)}(x,\rho){\tilde\Phi^{*'}}(x,\rho)-
\Phi^{(j-1)}(x,\rho){\tilde\vv^{*'}}(x,\rho), \\
\mathcal{P}_{j2}(x,\rho)=\Phi^{(j-1)}(x,\rho){\tilde\vv}^{*}(x,\rho)-
\vv^{(j-1)}(x,\rho){\tilde\Phi}^{*}(x,\rho).
\end{array}
\end{equation}

Applying Lemmas~\ref{lem:asympt} and 3$^*$ and using \eqref{eqh1}, 
for $\rho \in \Theta^+_{\de}$, $|\rho| \to \iy$, $x \in (0, \pi)$
we calculate
\begin{align*}
    \mathcal{P}_{11}(x, \rho) & = \frac{1}{2} P_{+}(x) (\tilde P^{\bu *}_{+}(0))^{-1} \tilde P^{\bu *}_{+}(x) +
    \frac{1}{2} P^{\bu}_{+}(x) (\tilde P^{*}_{+}(0))^{-1} \tilde P^{*}_{+}(x) + O(\rho^{-1}), \\
    \mathcal{P}_{12}(x, \rho) & = \frac{1}{2 i \rho} P_{+}^{\bu}(x) (P_{+}^{\bu}(0))^{-1} \tilde P_{+}^*(x) -
    \frac{1}{2 i \rho} P_{+}(x) (\tilde P_{+}^{\bu *}(0))^{-1}\tilde P_{+}^{\bu *}(x) + O(\rho^{-2}).
\end{align*}

Similarly, for $\rho \in \Theta^-_{\de}$, $|\rho| \to \iy$, $x \in (0, \pi)$
\begin{align*}
    \mathcal{P}_{11}(x, \rho) & = \frac{1}{2} P_{-}(x) (\tilde P^{\bu *}_{-}(0))^{-1} \tilde P^{\bu *}_{-}(x) +
    \frac{1}{2} P^{\bu}_{-}(x) (\tilde P^{*}_{-}(0))^{-1} \tilde P^{*}_{-}(x) + O(\rho^{-1}), \\
    \mathcal{P}_{12}(x, \rho) & = -\frac{1}{2 i \rho} P_{-}^{\bu}(x) (P_{-}^{\bu}(0))^{-1} \tilde P_{-}^*(x) +
    \frac{1}{2 i \rho} P_{-}(x) (\tilde P_{-}^{\bu *}(0))^{-1}\tilde P_{-}^{\bu *}(x) + O(\rho^{-2}).
\end{align*}

Using Lemma~\ref{lem:P2}, we obtain for $\rho \in \Theta^+_{\de} \cup \Theta^-_{\de}$, $|\rho| \to \iy$, $x \in (0, \pi)$, that
\begin{equation} \label{Pasimpt}
    \mathcal{P}_{11}(x, \rho) = \Omega(x) + O(\rho^{-1}), \quad 
    \mathcal{P}_{12}(x, \rho) = \rho^{-1} \Lambda(x) + O(\rho^{-2}),
\end{equation}
where
$$
    \Omega(x) := \frac{1}{2} \left( P_{-}(x) \tilde P^*_{+}(x) + P_{+}(x) \tilde P^*_{-}(x) \right), 
$$
$$
    \Lambda(x) := \frac{1}{2 i} \left( P_{-}(x) \tilde P^*_{+}(x) - P_{+}(x) \tilde P^*_{-}(x) \right).
$$
Substituting \eqref{Phiexp} and \eqref{Phiexp*} into \eqref{Pj12}, we get
$$
    \mathcal{P}_{11} = \vv \tilde S^{*'} - S\tilde \vv^{*'} + \vv (\tilde M^* - M) \tilde \vv^*,
$$
$$
    \mathcal{P}_{12} = S \tilde \vv^* - \vv \tilde S^{*} + \vv (M - \tilde M^*) \tilde \vv^*,    
$$
Note that for each fixed $x \in (0, \pi)$, the matrix functions $\vv$, $\vv^*$, $S$, $S^*$ and they derivatives with respect to $x$ 
are entire in $\rho$ of order $1$.
Since $\tilde M^*(\rho) \equiv \tilde M(\rho) \equiv M(\rho)$, it follows for each fixed $x \in (0, \pi)$
that the entries of $\mathcal{P}_{11}(x, \rho)$ and $\mathcal{P}_{12}(x, \rho)$ 
are entire functions in $\rho$ of order not greater than $1$.
Together with \eqref{Pasimpt} this yields $\mathcal{P}_{11}(x, \rho) \equiv \Omega(x)$, 
$\mathcal{P}_{12}(x, \rho) \equiv 0$, $\Lambda(x) \equiv 0$.

Differentiating $\Omega(x)$ and using \eqref{cauchyP}, we get
$$
    \Omega'(x) = i (\Lambda(x) \tilde Q_1(x) - Q_1(x) \Lambda(x) ) = 0
$$
for almost all $x$ in $[0, \pi]$. Therefore, $\Omega(x) \equiv \Omega(0) \equiv I$, $x \in [0, \pi]$.
Thus, $\mathcal{P}_{11}(x, \rho) = I$, $x \in (0, \pi)$. By virtue of~\eqref{defP} we have 
$\vv(x, \rho) \equiv \tilde \vv(x, \rho)$,
$\Phi(x, \rho) \equiv \tilde \Phi(x, \rho)$ and consequently, $L = \tilde L$.
\end{proof}

\medskip


\vspace{1cm}

Natalia Bondarenko

Department of Mathematics

Saratov State University

Astrakhanskaya 83, Saratov 410026, Russia

{\it bondarenkonp@info.sgu.ru}

\end{document}